\documentclass[a4paper,11pt]{amsart}
\usepackage[
  top=3cm,
  bottom=2cm,
  left=2cm,
  right=2cm,
  headheight=14pt,
  headsep=1cm,
  footskip=1cm
]{geometry}
\usepackage{amsmath, epsfig, verbatim}
\usepackage{amsfonts}
\usepackage{amsthm}
\usepackage{mathtools}
\usepackage{amsmath}
\usepackage{amssymb} 
\usepackage{graphics,graphicx}
\usepackage{epstopdf}
\usepackage{listings}
\usepackage{xcolor}
\usepackage{tikz}
\usepackage{subcaption}
\usepackage{pgfplots}
\usepackage{caption}
\usepackage{booktabs}
\usepackage{tabularx}
\usepackage[linesnumbered,ruled,vlined]{algorithm2e}

\allowdisplaybreaks
\definecolor{codegray}{rgb}{0.95,0.95,0.95}
\definecolor{pykeyword}{rgb}{0.13,0.13,1}
\definecolor{pystring}{rgb}{0.58,0,0.82}
\newcommand{\lr}[1]{\left(#1\right)}
\lstdefinestyle{pythonstyle}{
    backgroundcolor=\color{codegray},
    language=Python,
    basicstyle=\ttfamily\small,
    keywordstyle=\color{pykeyword}\bfseries,
    stringstyle=\color{pystring},
    commentstyle=\color{gray},
    showstringspaces=false,
    numbers=left,
    numberstyle=\tiny,
    frame=single,
    breaklines=true,
    tabsize=4,
}

\DeclarePairedDelimiter\floor{\lfloor}{\rfloor}
\numberwithin{equation}{section}
\theoremstyle{plain}
\newtheorem{theorem}{Theorem}
\newtheorem{defn}[theorem]{Definition}
\newtheorem{coro}[theorem]{Corollary}
\newtheorem{lemma}[theorem]{Lemma}
\begin{document}
\title{On the visibility polynomial of graphs} 
\author{Tonny K B}
\address{Tonny K B, Department of Mathematics, College of Engineering Trivandrum, Thiruvananthapuram, Kerala, India, 695016.}
\email{tonnykbd@cet.ac.in}
\author{Shikhi M}
\address{Shikhi M, Department of Mathematics, College of Engineering Trivandrum, Thiruvananthapuram, Kerala, India, 695016.}
\email{shikhim@cet.ac.in}
\begin{abstract}In a simple graph $G(V,E)$ and any subset $X$ of $V$, two vertices $u$ and $v$ are said to be $X$-visible if there exists a shortest path between them such that none of the internal vertices on the path belong to the set $X$. A set $X$ is called a mutual-visibility set of $G$ if every pair of vertices in $X$ is $X$-visible.  The visibility polynomial of a graph $G$ is defined as $\mathcal{V} (G)=\sum_{i\geq 0} r_i x^{i}$ where $r_i$ denotes the number of mutual-visibility sets in $G$ of cardinality $i$. In this paper, the visibility polynomial is studied for some well-known graph classes. In particular, the instance at which the number of maximal mutual-visibility sets is equal for cycle graphs is identified. The visibility polynomial of the join of two graphs is studied. The algorithm for computing the visibility polynomial of a graph has been identified to have a time complexity of $O(n^32^n)$, making the problem computationally intensive for larger graphs.\end{abstract}
\maketitle
\section{Introduction}

    Let $G(V,E)$ be a simple graph and let $X\subseteq V$.  Two vertices $u,v\in V$ are said to be $X-$visible \cite{Stefano} if there exists a shortest path $P$ from $u$ to $v$ such that $V(P)\cap X \subseteq \{u, v\}$. A set $X$ is called a mutual-visibility set of $G$ if every pair of vertices in $X$ is $X$-visible.
 The cardinality of the largest mutual-visibility set of $G$ is called the mutual-visibility number of $G$, denoted by $\mu(G)$, and the number of such mutual-visibility sets of $G$ is denoted by $r_{\mu}(G)$. The number of mutual-visibility sets of order $k$ of $G$  having diameter $d$ is denoted by $\Theta_{k,d}(G)$. 

    The concept of mutual visibility in graphs has garnered growing interest because of its relevance in a wide range of theoretical and applied domains. Wu and Rosenfeld studied visibility problems in pebble graphs in \cite{Geo_convex_1} and \cite{Geo_convex_2}. In \cite{Stefano}, Di Stefano defined the mutual-visibility set in the context of graph theory, which provides a structural framework to understand how information, influence, or coordination can be maintained in systems where communication is constrained by topology.

Mutual visibility plays a crucial role in robotics, particularly in multi-agent systems where agents (robots) must reposition themselves to ensure that every pair has an unobstructed line of sight. This leads to distributed algorithms for visibility-based formation control, surveillance, and navigation in unknown or dynamic environments. For recent studies, see \cite{robotics1,robotics2,robotics3,robotics4,robotics5,robotics6,robotics7}. Other emerging applications include sensor networks, where mutual visibility among sensors affects coverage and connectivity, and visibility graphs are used to study geometric arrangements and motion planning.

A related concept is that of a general position set, independently introduced in \cite{GP_1, GP_3}, which refers to a subset of vertices in which no three distinct vertices lie on a common geodesic. A subset $S\subseteq V$ of a connected graph $G$
 is said to be a general position set if no vertex in $S$ lies on a shortest path between two other vertices of
$S$. The general position polynomial of a graph is introduced and studied in \cite{GP_2}.

The graph-theoretic concept of mutual visibility has attracted considerable attention and has been studied extensively in a series of works \cite{MV_1, MV_2, MV_3, MV_4, MV_5, MV_6, MV_7, MV_8,MV_9,TMV_1,TMV_2}. Some variants of mutual visibility have been introduced in \cite{MV_10}. In \cite{sandi}, B. Csilla et. al. introduced the visibility polynomial of a graph $G$ as a polynomial invariant that encodes information about the number of mutual-visibility sets within the graph.

\textbf{Contribution:} The visibility polynomial for certain classes of graphs is studied in section 4. The instance in which the number of maximal mutual-visibility sets are equal for cycle graphs is identified.  In section 5, the visibility polynomial of the join of two graphs is studied.  Algorithmic results on finding the visibility polynomial of a graph are provided in section 6. With a time complexity of $O(n^32^n)$, computing the visibility polynomial is challenging, making theoretical analysis crucial for deeper insights and simplification.
\section{Notations and preliminaries}
In the present paper, $G(V,E)$ represents a simple and undirected graph with vertex set $V$ and edge set $E$. Unless otherwise stated, all graphs in this paper are assumed to be connected, so that there exists at least one path between every pair of vertices. We follow the standard graph-theoretic definitions and notation as presented in \cite{Harary}.

The complement of a graph $G$ is the graph $\overline{G}$ having the same set of vertices as that of $G$ and two vertices in $\overline{G}$ are adjacent if and only if they are not adjacent in $G$. $K_n$ denotes a complete graph on $n$ vertices and ${n\choose 2}$ edges in which there is an edge between any pair of distinct vertices. A sequence of vertices $(u_0,u_1,u_2,\ldots,u_{n})$ is referred to as a $(u_0,u_n)$-path if $u_iu_{i+1}\in E(G)$, $\forall i\in\{0,1,\ldots,(n-1)\}$.  $C_n$ denotes a cycle (or circuit) which is a path $(u_0,u_1,u_2,\ldots, u_n)$ together with an edge $u_0u_n$. A graph is bipartite if its vertex set can be partitioned into two subsets $X$ and $Y$ so that any edge of $G$ has one end vertex in $X$ and the other in $Y$. If each vertex of $X$ is joined to every vertex of $Y$ in a bipartite graph, it is called a complete bipartite graph, denoted by $K_{m,n}$. 
Two graphs $G$ and $H$ are isomorphic if there exists a bijection $\phi:V(G)\rightarrow V(H)$ such that $uv \in E(G)$ if and only if $\phi(u)\phi(v) \in E(H)$ for all $u, v \in V(G)$. We denote this by $G\cong H$.

 The distance between two vertices $ u $ and $ v $ in $G$  is denoted by $ d_G(u,v)$, which is the length of the shortest $ (u,v)$-path in $ G$. The maximum distance between any pair of vertices of $G$  is called the diameter of  $G$, denoted by $diam(G)$. Let $X\subseteq V(G)$. Then the induced subgraph $G[X]$ of $G$ by $X$ is the graph with vertex set $X$ and with the edges of $G$ having both endpoints in $X$. The diameter of $G[X]$ in $G$ is denoted by  $diam_G(X)$  and is defined by, $diam_G(X) = \displaystyle\max_{u, v \in X} d_G(u, v)$. 
 
 A clique is a subset of vertices in a graph where every pair of distinct vertices is connected by an edge. A clique is a complete subgraph within a graph. A clique with $k$ vertices is a $k-$clique and the number of $k-$cliques in a graph $G$ is denoted by $c_k(G)$.
 
The graph $G\backslash e$ denotes the graph $G(V, E-\{e\})$ derived from $G$ by deleting a single edge $e$ from $G$. If $G$ and $H$ are two graphs, then the join, $G \vee H$ is the graph with the vertex set $V(G)\cup V(H)$ and the edge set $E(G) \cup E(H)\cup \lbrace uv:u \in V(G), v\in V(H) \rbrace$. Note that, for any graphs $G$ and $H$, $diam(G\vee H)=2$.
\section{Visibility polynomial of graphs}
Let $G$ be a graph of order $n$. Then the visibility polynomial, $\mathcal{V}(G)$, of  $G$  is defined \cite{sandi} as
 $$\mathcal{V}(G)=\sum_{i\geq 0} r_i x^{i}$$  where $r_i$ denote  the number of mutual-visibility sets in $G$ of cardinality $i$. Note that the degree of visibility polynomial of $G$ is $\mu(G)$. Since every set of at most two vertices forms a mutual-visibility set, the visibility polynomial of a graph of order $n$ always begins with $1+nx+{n\choose 2}x^2+\cdots$

 Since mutual visibility depends on the adjacency relations in a graph, the visibility polynomial is an isomorphism invariant, meaning that isomorphic graphs always have the same visibility polynomial. But it is not a complete isomorphism invariant, since there exist non-isomorphic graphs that share the same visibility polynomial. For example, the non-isomorphic graphs in Figure \ref{non_isomorphic} have the same visibility polynomial; $\mathcal{V}(G_1)=\mathcal{V}(G_2)=1+4x+6x^2+4x^3$. 
\begin{figure}[h]
    \centering
\begin{tikzpicture}[scale=1.5]
  \begin{scope}[xshift=-2cm]
    \node[circle, fill=black, inner sep=2pt] (G1) at (0,0) {};
    \node[circle, fill=black, inner sep=2pt] (G2) at (2,0) {};
    \node[circle, fill=black, inner sep=2pt] (G3) at (2,1) {};
    \node[circle, fill=black, inner sep=2pt] (G4) at (0,1) {};

     \draw[line width=1pt] (G1) -- (G2);
    \draw[line width=1pt] (G2) -- (G3);
    \draw[line width=1pt] (G3) -- (G4);
    \draw[line width=1pt] (G4) -- (G1);
    \draw[line width=1pt] (G4) -- (G2);
    \node at (1,-0.3) {$G_1$};
  \end{scope}

  \begin{scope}[xshift=2cm]
    \node[circle, fill=black, inner sep=2pt] (H1) at (0,0) {};
    \node[circle, fill=black, inner sep=2pt] (H2) at (2,0) {};
    \node[circle, fill=black, inner sep=2pt] (H3) at (2,1) {};
    \node[circle, fill=black, inner sep=2pt] (H4) at (0,1) {};
    
    \draw[line width=1pt] (H1) -- (H2);
    \draw[line width=1pt] (H2) -- (H3);
    \draw[line width=1pt] (H3) -- (H4);
    \draw[line width=1pt] (H4) -- (H1);
    \node at (1,-0.3) {$G_2$};
  \end{scope}
   \end{tikzpicture}
 \caption{Non-isomorphic graphs with the  same visibility polynomial}
    \label{non_isomorphic}
\end{figure}
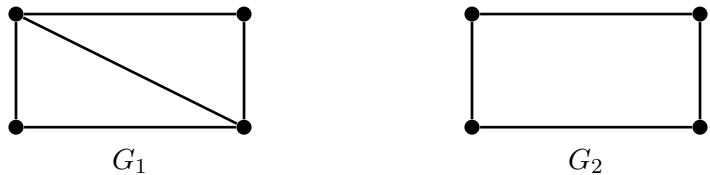

Using a Python program and high performance computing facilities, we identified all non-isomorphic graphs of order up to 9 that have identical visibility polynomials. In Table \ref{ch2.tbl1}, we summarize the size of the largest such groups and its visibility polynomials. For example, among 11,117 non-isomorphic graphs, only up to 14 possess identical visibility polynomials,  indicating that the visibility polynomial is a highly discriminative invariant for graph comparison. In Table \ref{ch2.tbl1}, we use the following notations: $n$ denotes the order of the graph, $T(n)$ represents the total number of non-isomorphic graphs of order $n$, and $M(n)$ indicates the maximum number of graphs sharing the same visibility polynomial.
 \begin{table}[htbp]
    \centering
    \caption{Non-isomorphic graphs with identical visibility polynomial}\label{ch2.tbl1}
    \begin{tabularx}{0.9\linewidth}{@{}XXXl@{}}
        \toprule
       n & T(n) & M(n) & Polynomial \\
        \midrule
        4 & 6 & 2& $1+4x+6x^2+4x^3$ \\
        5 & 21 & 2 &  $1+5x+10x^2+7x^3$\\
        6 & 112 & 4 & $1+6x+15x^2+14x^3+3x^4$ \\
        7 & 853 & 6 & $1+7x+21x^2+26x^3+9x^4$ \\
        8 & 11117 & 14 & $1+8x+28x^2+52x^3+46x^4+12x^5$ \\
    9 & 261080 & 59 & $1+ 9x+ 36x^2+82x^3+107x^4+66x^5+13x^6$\\
        \bottomrule
    \end{tabularx}
\end{table}

\begin{theorem}\label{theorem19}
    Let $G$ be a disconnected graph with 2 components $G_1$ and $G_2$, then  $\nu(G)= \nu(G_1)+\nu(G_2)-1$.
\end{theorem}
\begin{proof}
Because $G$ is disconnected, with 
$G_1$ and $G_2$ as its components, any mutual-visibility set must be entirely contained within one of the components. No such set can include vertices from both $G_1$ and $G_2$, as there are no paths connecting them. Hence, the visibility polynomial of $G$ is obtained by adding the visibility polynomials of 
$G_1$ and $G_2$, and then subtracting 1 to account for the fact that the empty set is counted twice. Hence the result.
\end{proof}
This result may be generalized in the case of a disconnected  graph with $m $ components $G_1,G_2,\ldots,G_m$ as follows:
\begin{theorem}\label{theorem122}
If $G$ is disconnected graph with $m$ components $G_1, G_2, \ldots ,G_m$, then $\mathcal{V}(G)= \mathcal{V}(G_1)+\mathcal{V}(G_2)+\cdots+\mathcal{V}(G_m)-m+1$. 
\end{theorem}
In \cite{sandi}, B. Csilla et al. showed that the visibility polynomial of a path graph of order $n \geq 2$ is given by
$\mathcal{V}(P_n)= 1+nx+{n \choose 2}x^2$.
\begin{coro}
     Let $e\in V(P_n)$ be such that $P_n\setminus e\cong P_{n_1}\cup P_{n_2}$. Then $\mathcal{V}(P_n\setminus e)=\mathcal{V}(P_n)-n_1n_2x^2$.
\end{coro}
\begin{proof}
    Since $P_n\setminus e\cong P_{n_1}\cup P_{n_2}$, it follows from Theorem \ref{theorem19} that $\mathcal{V}(P_n\setminus e)=\mathcal{V}(P_{n_1}\cup P_{n_2})=\mathcal{V}(P_{n_1})+\mathcal{V}(P_{n_2})-1=\left(1+n_1x+{n_1\choose 2}x^2\right)+\left(1+n_2x+{n_2\choose 2}x^2\right)-1=1+nx+{n\choose 2}x^2-n_1n_2x^2=\mathcal{V}(P_n)-n_1n_2x^2$.
    Hence, the result.
\end{proof}
\section{Visibility polynomial of some graph classes}
\begin{theorem}\label{theorem12}  The visibility polynomial of the complete graph of order $n$ is given by
$\mathcal{V}(K_n)= (1+x)^n$.
\end{theorem}
\begin{proof}
Let $G(V,E)$ be a complete graph of order $n$. For $1\leq i \leq n$, consider any subset $S=\{v_1,v_2,\ldots, v_i\}$ of $V$. Then the shortest path between any two vertices $v_j,v_k\in S$ is the edge $v_jv_k$. Since the path has no internal vertices, $S$ becomes a mutual-visibility set. Since we have ${n\choose i}$ choices for $S$, $r_i={n\choose i}$. Hence, $\mathcal{V}(K_n)= \displaystyle \sum_{i=0}^n{n \choose i} x^i=(1+x)^n$. 
\end{proof}
\begin{theorem}
    The visibility polynomial of a star graph $S_n,\ (n\geq 0)$ of order $n+1$ is given by $\mathcal{V}(S_n)=x+nx^2+(1+x)^n.$
\end{theorem}
\begin{proof}
Here $r_1=n+1$ and $r_2={n+1\choose 2}$. Let $T=\{v_1,v_2,\ldots,v_n\}$ be the set of all pendant vertices of $S_n$ and let $v_{n+1}$ be the centre vertex of degree $n$. Since any shortest path from $v_i$ to $v_j$, where $i,j\in \{1,2,\ldots,n\}$ includes the vertex $v_{n+1}$, $v_{n+1}$ cannot be included in any mutual-visibility set of cardinality greater than 2. Let $S$ be any subset of $T$ having cardinality $i>2$. Then the shortest path between any two vertices $v_i$ and $v_j$ of $S$ is $(v_i,v_{n+1},v_j)$, which includes no other vertices of $S$ other than $v_i$ and $v_j$ so that $S$ is a mutual-visibility set in $G$. Hence, the number of mutual-visibility sets of cardinality $i>2$ is given by ${n\choose i}$. Hence, it follows that 
\begin{align*}
    \mathcal{V}(S_n)&=1+(n+1)x+{n+1\choose 2}x^2+\sum_{i=3}^n{n\choose i}x^i\\
    &=1+nx+x+\left[{n\choose 2}+{n\choose 1}\right]x^2+\sum_{i=3}^n{n\choose i}x^i\\
   & =x+nx^2+\sum_{i=0}^n{n\choose i}x^i=x+nx^2+(1+x)^n
\end{align*}
Hence the proof.
\end{proof}
\begin{theorem}\label{ch2.th7}The visibility polynomial of a cycle of order $n \geq 3$ is given by
$\mathcal{V}(C_n)=1+nx+{n \choose 2}x^2 +r_3x^3$ where
$$r_3 = \left\{
  \begin{array}{ll}
  \dfrac{ n(n^2-1) }{24}& \text{if } n \text{ is odd} \\[10pt]
  \dfrac{(n-2)n(n+8)}{24} & \text{if } n \text{ is even}
  \end{array}
\right.$$
\end{theorem}
\begin{proof}
Since $\mu(C_n)=3$ for $n\geq 3$ (See \cite{Stefano}), it is enough to find $r_3$.
To construct a mutual-visibility set of size 3, we label the vertices of $C_n$ as $\{v_1,v_2,\ldots, v_n\}$ such that $C_n=(v_1,v_2,\ldots,v_n)$. Let $M=\{v_i,v_j,v_k\}$ be a mutual-visibility set of size 3 in $C_n$. We will count the number of possible choices of $i,j$ and $k$ so that $M$ forms a mutual-visibility set. Assume that $i<j<k\leq n$, to avoid duplicate counting.

First, we consider the case when $n$ is even.
If $i>\frac{n}{2}$, then $j,k\in\left\{i+1,i+2,\ldots,n\right\}$ and the shortest path from $v_i$ to $v_k$ includes the vertex $v_j$ which is a contradiction, since $M$ is a mutual-visibility set. If $j>\frac{n}{2}+i$, then the shortest path from $v_i$ to $v_j$ includes the vertex $v_k$ which is again a contradiction. Hence, the possible ranges of $i$ and $j$ are  $1\leq i \leq \frac{n}{2}$ and $i+1\leq j \leq \frac{n}{2}+i$ respectively.

Now we consider the following cases of $i,j$ and will count the possible choices of $k$ so that $M$ forms a mutual visibility set.\\
Case 1: Let $1\leq i\leq \frac{n}{2}-1$ and  $i+1\leq j\leq \frac{n}{2}$. 
     We claim that $k\in\left\{\frac{n}{2}+i,\frac{n}{2}+i+1,\ldots, \frac{n}{2}+j\right\}$. Here, the shortest path from $v_i$ to $v_j$ becomes $(v_i,v_{i+1},v_{i+2},\ldots, v_j)$ which doesn't include the vertex $v_k$, the shortest path from $v_j$ to $v_k$ is $(v_j,v_{j+1},v_{j+2},\ldots, v_k)$ which doesn't include $v_i$ and the shortest path from $v_i$ to $v_k$ given by $(v_i,v_{i-1},\ldots ,v_1,v_n,v_{n-1},\ldots ,v_k)$  which doesn't include $v_j$. Hence $\{v_i,v_j,v_k\}$ becomes a mutual-visibility set. 

     If $k< \frac{n}{2}+i$, then the shortest path from $v_i$ to $v_k$ includes the vertex $v_j$ and if $k> \frac{n}{2}+j$, then the shortest path from $v_j$ to $v_k$ includes the vertex $v_i$. Hence the claim. \\
    Case 2: Let $1\leq i\leq \frac{n}{2}-1$ and  $\frac{n}{2}+1\leq j\leq \frac{n}{2}+(i-1)$. 
      We claim that $k\in\left\{\frac{n}{2}+i,\frac{n}{2}+i+1,\ldots, n\right\}$. Here, the shortest path from $v_i$ to $v_j$ becomes $(v_i,v_{i+1},v_{i+2},\ldots, v_j)$ which does not include the vertex $v_k$, the shortest path from $v_j$ to $v_k$ is $(v_j,v_{j+1},v_{j+2},\ldots, v_k)$ which does not include $v_i$ and the shortest path from $v_i$ to $v_k$ given by $(v_i,v_{i-1},\ldots, v_1,v_n,v_{n-1},\ldots, v_k)$ which does not include $v_j$. Hence $\{v_i,v_j,v_k\}$ becomes a mutual-visibility set. 
      
      If $k< \frac{n}{2}+i$, then the shortest path from $v_i$ to $v_k$ includes the vertex $v_j$ and the maximum possible value of $k$ is $n$ only. Hence, the claim. \\
      Case 3: Let $1\leq i\leq \frac{n}{2}-1$ and $j=\frac{n}{2}+i$.
      We claim that $k\in\left\{\frac{n}{2}+i+1,\frac{n}{2}+i+2,\ldots, n\right\}$. Here, the proof of the claim is exactly the same as in Case 2, but we exclude the case that $k=\frac{n}{2}+i$ since the vertex $v_{\frac{n}{2}+i}$ is already selected. \\
      Case 4: Let $i=\frac{n}{2}$.
      Then the possible choices of $j$ are $\frac{n}{2}+1,\frac{n}{2}+2,\ldots,n-1$ and the only possible choice for $k$ is $n$ only. Since, if $k\in \{j,j+1,\ldots,n- 1\}$, then the shortest path from $v_i$ to $v_k$ includes the vertex $v_j$.

 From the above cases, it follows that,
\begin{align*}
  r_3 &=\sum_{i=1}^{\frac{n}{2}-1}\left[\sum_{j=i+1}^{\frac{n}{2}} (j-i+1)+\sum_{j=\frac{n}{2}+1}^{\frac{n}{2}+i-1}\left(\frac{n}{2}-i+1\right)+\lr{\frac{n}{2}-i}\right]+\left(\frac{n}{2}-1\right)\\
  &=\sum_{i=1}^{\frac{n}{2}-1}\left[\sum_{k=2}^{\frac{n}{2}-i+1} k+\sum_{j=\frac{n}{2}+1}^{\frac{n}{2}+i-1}\left(\frac{n}{2}-i+1\right)+\lr{\frac{n}{2}-i}\right]+\left(\frac{n}{2}-1\right) \quad \text{( Let } k = j-i+1\text{)} \\
 &= \dfrac{1}{8}\sum_{i=1}^{\frac{n}{2}-1}\left[(n^2+6n-8)-4i-4i^2\right]+\left(\dfrac{n}{2}-1\right) \\
 &=\dfrac{1}{8}\left[\dfrac{(n-2)(n^2+6n-8)}{2}-\dfrac{(n-2)n}{2}-\dfrac{(n-2)(n-1)n}{6} \right]+\left(\dfrac{n}{2}-1\right)\\[3pt]
  &=\frac{n(n^2 + 6n - 16)}{24}=\dfrac{(n-2)n(n+8)}{24}
\end{align*} 
Now, we consider the case where $n$ is odd. If $i>\floor*{\frac{n}{2}}$, then $j,k\in\left\{i+1,i+2,\ldots, n\right\}$ and the shortest path from $v_i$ to $v_k$ includes the vertex $v_j$ which is a contradiction, since $M$ is a mutual-visibility set. If $j>\floor*{\frac{n}{2}}+i$, then the shortest path from $v_i$ to $v_j$ includes the vertex $v_k$, which is again a contradiction. Hence the possible ranges of $i$ and $j$ are  $1\leq i \leq \floor*{\frac{n}{2}}$ and $i+1\leq j \leq \floor*{\frac{n}{2}}+i$.

Now we consider the following cases of $i,j$ and will count the possible choices of $k$.\\
Case 1: Let $1\leq i\leq \floor*{\frac{n}{2}}$ and  $i+1\leq j\leq \floor*{\frac{n}{2}}+1$.\\
     We claim that $k\in\left\{\floor*{\frac{n}{2}}+i+1,\floor*{\frac{n}{2}}+i+2,\ldots, \floor*{\frac{n}{2}}+j\right\}$. Here, the shortest path from $v_i$ to $v_j$ becomes $(v_i,v_{i+1},v_{i+2},\ldots, v_j)$ which does not include the vertex $v_k$, the shortest path from $v_j$ to $v_k$ is $(v_j,v_{j+1},v_{j+2},\ldots, v_k)$ which does not include $v_i$, and there exists a shortest path from $v_i$ to $v_k$ given by $(v_i,v_{i-1},\ldots, v_1,v_n,v_{n-1},\ldots, v_k)$ which does not include $v_j$. Hence, $\{v_i,v_j,v_k\}$ becomes a mutual-visibility set. 

     If $k< \floor*{\frac{n}{2}}+i+1$, then the shortest path from $v_i$ to $v_k$ includes the vertex $v_j$ and if $k> \floor*{\frac{n}{2}}+j$, then the shortest path from $v_j$ to $v_k$ includes the vertex $v_i$. Hence, the claim. \\
     Case 2: Let $1\leq i\leq \floor*{\frac{n}{2}}$ and  $\floor*{\frac{n}{2}}+2\leq j\leq \floor*{\frac{n}{2}}+i$. \\
      We claim that $k\in\left\{\floor*{\frac{n}{2}}+i+1,\floor*{\frac{n}{2}}+i+2,\ldots, n\right\}$. Here, the shortest path from $v_i$ to $v_j$ becomes $(v_i,v_{i+1},v_{i+2},\ldots, v_j)$ which does not include the vertex $v_k$, the shortest path from $v_j$ to $v_k$ is $(v_j,v_{j+1},v_{j+2},\ldots, v_k)$ which does not include $v_i$, and there exists a shortest path from $v_i$ to $v_k$ given by $(v_i,v_{i-1},\ldots, v_1,v_n,v_{n-1},\ldots, v_k)$ which does not include $v_j$. Hence $\{v_i,v_j,v_k\}$ becomes a mutual-visibility set. 
      
      If $k\leq\floor*{\frac{n}{2}}+i$, then the shortest path from $v_i$ to $v_k$ includes the vertex $v_j$ and the maximum possible value of $k$ is $n$. Hence, the claim. 
 From the above cases, it follows that,
\begin{align*}
     r_3 &=\sum_{i=1}^{\floor*{\frac{n}{2}}}\left[\sum_{j=i+1}^{\floor*{\frac{n}{2}}+1} (j-i)+\sum_{j=\floor*{\frac{n}{2}}+2}^{\floor*{\frac{n}{2}}+i}\left(\floor*{\frac{n}{2}}-i+1\right)\right]\\
  &=\sum_{i=1}^{\floor*{\frac{n}{2}}}\left[\sum_{k=1}^{\floor*{\frac{n}{2}}-i+1} k+(i-1)\lr{\floor*{\frac{n}{2}}-i+1}\right] \quad \text{( Let } k = j-i\text{)}\\
  &= \sum_{i=1}^{\floor*{\frac{n}{2}}}
\left[
\dfrac{\left(\floor*{\dfrac{n}{2}} - i + 1\right)\left(\floor*{\dfrac{n}{2}} - i + 2\right)}{2}
+ (i - 1)\left(\floor*{\dfrac{n}{2}} - i + 1\right)
\right] \\
&= \sum_{m=1}^{\floor*{\frac{n}{2}}}
\left[
\frac{m(m + 1)}{2}
+ \left(\floor*{\dfrac{n}{2}} - m\right)m
\right]
\quad \text{( Let } m = \floor*{\dfrac{n}{2}} - i + 1\text{)} \\
&= \frac{\left\lfloor \dfrac{n}{2} \right\rfloor \left( \left\lfloor \dfrac{n}{2} \right\rfloor + 1 \right) \left(2 \left\lfloor \dfrac{n}{2} \right\rfloor + 1 \right)}{6}\\
&=\dfrac{\lr{\dfrac{n-1}{2}}\lr{\dfrac{n+1}{2}}n}{6}=\dfrac{ n(n^2-1) }{24}
\end{align*}
Hence the proof.
 \end{proof}
Next, we will discuss some properties of the number of maximal mutual-visibility sets of $C_n$.
\begin{lemma}\label{ch2.lem1}
 (i) $\{r_{\mu}(C_n) : n=1,3,5,\ldots\}$ is an increasing sequence.
(ii) $\{r_{\mu}(C_n) : n=2,4,6,\ldots\}$ is an increasing sequence.
\end{lemma}
\begin{proof}
(i) From Theorem \ref{ch2.th7}, $r_{\mu}(C_n)= \frac{ n(n^2-1) }{24}$ if $n$ is an odd integer. Consider $f(x)=\frac{ x(x^2-1) }{24}, \ \ x>2$, then $f'(x)=\frac{ 3x^2-1 }{24} $ and $3x^2-1$ is positive for $x>2$, which implies that $f(x)$ increases when $x>2$. Therefore,  $r_{\mu}(C_n)= \frac{ n(n^2-1) }{24}$ increases for odd integers $n$.\\
(ii) From Theorem \ref{ch2.th7}, $r_{\mu}(C_n)= \frac{(n-2)n(n+8)}{24}$ if $n$ is an even integer. Consider $f(x)=\frac{ (x-2)x(x+8) }{24}, \ \ x>3$, then $f'(x)=\frac{ 3x^2+12x-16 }{24} $. $f'(x) >0$ for $x>3$ which implies that $f(x)$ increases when $x>3$. Therefore, $r_{\mu}(C_n)= \frac{(n-2)n(n+8)}{24}$ increases for even integers $n$.
\end{proof}
\begin{lemma}\label{ch2.lem8}
 $r_{\mu}(C_{n})=r_{\mu}(C_{n+1})$ if and only if $n=6$.   
\end{lemma}
\begin{proof}
Consider the case where $n$ is even. Let $n=2k$, $k>1$. Then, $r_{\mu}(C_{n})=r_{\mu}(C_{n+1})$ if and only if $\frac{(2k-2)2k(2k+8)}{24} =\frac{ (2k+1)((2k+1)^2-1) }{24}$  if and only if $k=3$. When $n$ is odd, let $n=2k-1$, where $k>1$. Then $r_{\mu}(C_{n})=r_{\mu}(C_{n+1})$ if and only if $\frac{ (2k-1)((2k-1)^2-1) }{24} =\frac{(2k-2)2k(2k+8)}{24}$ if and only if $k=1$. That is a contradiction to the assumption that $k >1$. Therefore, $r_{\mu}(C_{n}) \neq r_{\mu}(C_{n+1})$ if $n$ is an odd integer. Hence, the result follows.
\end{proof}
\begin{lemma}\label{ch2.lem2}
  (i) If $n$ is odd, $r_{\mu}(C_n) < r_{\mu}(C_{n+1}) $. (ii) If $n$ is even and greater than $7$ , then
     $r_{\mu}(C_n) > r_{\mu}(C_{n+1})$. 
\end{lemma}
\begin{proof}
(i) If $n$ is odd, let $n=2k-1$, where $k\in \mathbb{N}$ and $k>1$. Then $r_{\mu}(C_{n})=\frac{ (2k-1)((2k-1)^2-1) }{24}=\frac{ (2k-1)(4k^2-4k) }{24}$ and $r_{\mu}(C_{n+1})=\frac{(2k-2)2k(2k+8)}{24}$.
Here, the result follows since $(2k-2)(k+4)-(2k-1)(k-1)=9k-9>0$ when $k>1$.\\
(ii) If $n$ is even and greater than $7$, let $n=2k$, where $k\in \mathbb{N}$ and $k>3$. Then $r_{\mu}(C_{n})=\frac{(2k-2)2k(2k+8)}{24}$ and $r_{\mu}(C_{n+1})=\frac{ (2k+1)4k(k+1) }{24}$. Here, the result follows since $(2k-2)(k+4)-(2k+1)(k+1)= 3k-9>0$ when $k>3$.    
\end{proof}
\begin{theorem}
    $r_{\mu}(C_{n_1})=r_{\mu}(C_{n_2})$ if and only if $n_1=6$ and $n_2=7$.
\end{theorem}
\begin{proof}
    Assume that $n_1 < n_2$. By Lemma \ref{ch2.lem1}, $r_{\mu}(C_n)$ increases for odd integers $n$. Therefore, they are all distinct. Similarly, $r_{\mu}(C_n)$ also increases and is distinct for even integers $n$. By Theorem \ref{theorem14}, $r_{\mu}(C_3)=1, r_{\mu}(C_4)=4, r_{\mu}(C_5)=5, r_{\mu}(C_6)=14$ and $ r_{\mu}(C_7)=14$. By Lemma \ref{ch2.lem1} and Lemma \ref{ch2.lem2}, it follows that $14\leq r_{\mu}(C_n) < r_{\mu}(C_{n+p})$  for all odd integers $n$ greater than or equal to 7 and for all $p \in \mathbb{N}$.  Therefore, $r_{\mu}(C_{n_1})\neq r_{\mu}(C_{n_2})$ if $n_1$ is an odd integer and $n_1,n_2\geq 7$.
    
     If $n$ is an even integer greater than 7, by Lemma \ref{ch2.lem1}, $ r_{\mu}(C_n) \neq  r_{\mu}(C_{n+2p})$ for all $p \in \mathbb{N}$. Also, by Lemma \ref{ch2.lem8}, $ r_{\mu}(C_n) \neq  r_{\mu}(C_{n+1})$. Now let $n=2k$ and consider, $$r_{\mu}(C_{n+3})-r_{\mu}(C_{n})= \dfrac{(2k+3)((2k+3)^2-1)}{24}-\dfrac{(2k-2)2k(2k+8)}{24}=\dfrac{1}{24}(12k^2 + 84k + 24)$$
 $12k^2 + 84k + 24 >0$ for all $k \in \mathbb{N}$. Therefore, $14<r_{\mu}(C_{n}) < r_{\mu}(C_{n+3})$ for all even integers $n$ greater than 7. From Lemma \ref{ch2.lem2}, $r_{\mu}(C_{n+3}) < r_{\mu}(C_{n+(2p+1)})$ for all $p=2,3,4 \ldots$. That is $r_{\mu}(C_{n}) \neq r_{\mu}(C_{n+(2p+1)})$ for all $p=1,2,3 \ldots$. Therefore,  $r_{\mu}(C_{n_1})\neq r_{\mu}(C_{n_2})$ if $n_1$ is an even integer and $n_1,n_2\geq 7$. By Lemma \ref{ch2.lem1}, the result follows.
\end{proof}
\begin{lemma}\label{ch2.lem3}
    Let $G(V, E)$ be a complete bipartite graph with bipartition $(A,B)$. If $X$ is a subset of $V$ which does not contain at least one vertex from each of the partitions $A$ and $B$, then $X$ is a mutual-visibility set of $G$.
\end{lemma}
\begin{proof}
   By the assumption, there exist two vertices $a\in A$ and $b\in B$ such that $a, b \notin X$.  Let $u, v \in X$. Suppose that both $u$ and $v$ belong to the same partition. Consider the case $u, v \in A$. Since $G$ is a complete bipartite graph, $(u, b, v)$ is a shortest path from $u$ to $v$. Similarly, if $u, v \in B$, then $(u, a, v)$ is a shortest path from $u$ to $v$. In both cases $u$ and $v$ are $X$-visible.

   Next, we assume that $u$ and $v$ belong to different partite sets. Suppose $u \in A$ and $v \in B$. Then the shortest path from $u$ to $v$ is the edge $uv$. Therefore, $u$ and $v$ are $X$-visible. Hence, $X$ is a mutual-visibility set of $G$.
\end{proof}
\begin{lemma}\label{ch2.lem4}
 Let $G(V, E)$ be a complete bipartite graph with bipartition $(A,B)$, where $|A|=m$ and $|B|=n \geq 2$. Let $X=A \cup Y$ and $Y \subseteq B $. $X$ is a mutual-visibility set of $G$ if and only if $Y$ contains at most one element.
\end{lemma}
\begin{proof}
Assume that $X= A \cup Y$ is a mutual-visibility set of $G$, where $Y \subseteq B$. Suppose that $Y$ contains more than one vertex. Consider two vertices $b_1, b_2 \in Y$. Since $b_1$ and $b_2$ are not adjacent, any shortest path from $b_1$ to $b_2$ is of the form $(b_1, u, b_2)$, where $u \in A$. Therefore, the pair $\{b_1, b_2\}$ is not $X$-visible. This is a contradiction. Therefore, $Y$ contains at most one element. 

To prove the sufficiency of the result we consider the following cases.\\
Case 1: Let $Y= \phi$, then $X=A$. A shortest path between two vertices $u, v \in A$ is of the form $(u, b, v)$, where $b \in B$. So every pair $\{u, v\} \subseteq A$ is $X$-visible. Therefore, $X$ is a mutual-visibility set.\\   
Case 2: Let $Y$ contains only one vertex $b$. Choose an element $c$ from $B$ other than $b$. Then the shortest path between two vertices $u, v \in A$ is of the form $(u, c, v)$. So every pair $\{u, v\} \subseteq A$ is $X$-visible. All pairs of the form $\{u, b\}$, where $u \in A$ are $X$-visible, since $u$ and $b$ are adjacent. Therefore, $X$ is a mutual-visibility set.
\end{proof}

In \cite{Stefano}, G. D. Stefano identified  the size of the largest mutual visibility set of $K_{m,n}$ as follows:
\begin{theorem}[\cite{Stefano}]\label{ch2.th2}
    Let $G$ be a complete bipartite graph $K_{m,n}$ such that $m \geq 3$ and $n \geq 3$. Then $ \mu(G ) = m + n - 2$.
\end{theorem}
\begin{theorem}\label{ch2.th5}The visibility polynomial of the  complete bipartite graph $K_{m,n}(V, E)$ with bipartition $(A,B)$, where $|A|=m\geq 3 $ and $|B|=n \geq 3 , \ m\leq n$ is given by
  $\mathcal{V}(K_{m,n})=\sum_{i=0}^{m+n-2}r_ix^i$ where 
$$r_i=\left\{
  \begin{array}{ll}
  {m+n \choose i}& \text{if } 1\leq i \leq m+1 \\[8pt]
 {m+n \choose i}-{n \choose i-m}  & \text{if }  m+2\leq i \leq n+1 \\[8pt]
 {m+n \choose i}-{n \choose i-m}-{m \choose i-n}  & \text{if }  n+2\leq i \leq m+n-2
  \end{array}
\right.$$
\end{theorem}
\begin{proof}
Let $X\subseteq V$ be any subset of cardinality $i$. We consider various cases based on the value of $i$, and for each case, we count the number of instances in which $X$ forms a mutual-visibility set of cardinality $i$.\\
Case 1: Let $1\leq i\leq m+1$. In this case, $X$ falls under one of the conditions mentioned below:\\
    1. There exists at least one vertex $u\in A$ and at least one vertex $v\in B$ such that $u,v\notin X$.\\
    2. $X=A \cup Y$, where $Y \subseteq B$ contains at most one vertex.\\
    3. $X=Y'\cup B$, where $Y' \subseteq A$ contains at most one vertex.\\
   By Lemma \ref{ch2.lem3} and Lemma \ref{ch2.lem4}, it follows that $X$ is a mutual-visibility set of $G$. Hence, in this case, $K_{m,n}$ has ${m+n \choose i}$ mutual-visibility sets of order $i$.\\ Case 2: Let  $m+2\leq i \leq n+1$ (Only when $m<n$).  If $X \cap A \subsetneq A$, then $X \cap B$  is either $B$ or a proper subset of $B$. In the former case, $ X \cap A$ contains at most one element and hence by Lemma \ref{ch2.lem4}, $X$ is a mutual-visibility set of $G$. In the latter case, $X$ does not contain at least one of the vertices from each of the partitions $A$ and $B$, and it follows from Lemma \ref{ch2.lem3} that $X$ is a mutual-visibility set of $G$.

On the other hand, if $X \cap A = A$, then $X \cap B$ contains at least two vertices. Therefore, by Lemma \ref{ch2.lem4}, $X$ is not a mutual-visibility set of $G$. Since there are ${n \choose i-m}$ ways to choose $i-m$ elements from $B$, it follows that the number of mutual-visibility sets of $K_{m,n}$ having cardinality $i$ in this case is equal to ${m+n \choose i}-{n \choose i-m}$.\\
Case 3: Let $n+2\leq i \leq m+n-2$. 
If $X \cap A \subsetneq A$, then $X \cap B$  is either $B$ or a proper subset of $B$. In the former case, $ X \cap A$ contains more than one element, and hence, by Lemma \ref{ch2.lem4}, $X$ is not a mutual-visibility set of $G$. The number of such sets is ${m \choose i-n}$. In the latter case, $X$ does not contain at least one vertex from each of the partitions $A$ and $B$, and it follows from Lemma \ref{ch2.lem3} that  $X$ is a mutual-visibility set of $G$.

If $X \cap A=A$, then as proved in Case 2, $X$ is not a mutual-visibility set of $G$. The number of such sets is ${n \choose i-m}$. Therefore, the number of mutual-visibility sets of cardinality $i$ in this case is equal to ${m+n \choose i}-{n \choose i-m}-{m \choose i-n}$.\\
Case 4: If $i > m+n-2$, then by Theorem \ref{ch2.th2}, there is no mutual-visibility set of $G$ having cardinality $i$. This completes the proof.
\end{proof}
  \begin{coro}\label{theorem15}The visibility polynomial of the  complete bipartite graph $K_{n,n}$ is given by
  $\mathcal{V}(K_{n,n})=\displaystyle\sum_{i=0}^{2n-2}r_ix^i$ where 
$$r_i=\left\{
  \begin{array}{ll}
  {2n \choose i}& \text{if } 0\leq i \leq n+1 \\[8pt]
 {2n \choose i}-2{n \choose i-n}  & \text{if } n+2\leq i \leq 2n-2
  \end{array}
\right.$$
\end{coro}
\begin{proof}
    The result follows from case 1 and 3 of Theorem \ref{ch2.th5}. Note that  case 2 of the Theorem is trivially ruled out.
\end{proof}
 \section{Visibility polynomial of the join of two graphs}
\begin{theorem}\label{ch2.th3}
    If $G$ and $H$ are two complete disjoint graphs, then $\mathcal{V}(G\vee H)=\mathcal{V}(G).\mathcal{V}(H)$. 
\end{theorem}
\begin{proof}
   Let $G$ and $H$ be the complete graphs $K_m$ and $K_n$, respectively. Since $K_m\vee K_n \cong K_{m+n}$, from Theorem \ref{theorem12}, it follows that 
      $$ \mathcal{V}(K_m\vee K_n)=\mathcal{V}(K_{m+n})
       =(1+x)^{m+n}
       =(1+x)^m.(1+x)^n
       =\mathcal{V}(K_m).\mathcal{V}(K_n)$$
   Hence the proof.
\end{proof}

In \cite{Stefano}, G.D. Stefano characterized graphs with $\mu(G)=|V|$ as follows:
\begin{lemma}[\cite{Stefano}]\label{lemma_23}
Let $G=(V,E)$ be a graph such that $|V|=n$. Then $\mu(G)=|V|$ if and only if $G\cong K_n$.
\end{lemma}
\begin{coro}\label{ch2.coro1}
 Let $G$ and $H$ be disjoint graphs. $V(G \cup H)$ is a mutual-visibility set in $G \vee H$ if and only if $G$ and $H$ are complete. 
\end{coro}
\begin{proof}
    If $V(G \cup H)$ is a mutual-visibility set in $G \vee H$, by Lemma \ref{lemma_23}, $G \vee H$ is complete. Therefore, $G$ and $H$ are complete.  The converse follows from Theorem \ref{ch2.th3}.   
\end{proof}
\begin{lemma}\label{ch2.lem6}
    Let $G$ and $H$ be disjoint graphs. If $X=A \cup B$, where $A \subsetneq V(G)$ and $B \subsetneq V(H)$, then $X$ is a mutual-visibility set of the join $G \vee H$.
\end{lemma}
\begin{proof}
    Let $u, v \in X$. If $u$ and $v$ belong to different sets $A$ and $B$, then they are adjacent. On the other hand, assume that $u, v \in A$. If they are adjacent, they are $X$-visible. Otherwise, there exists a $(u,v)$-path $(u,h,v)$, where $h \in V(H) \setminus B$. Similarly, if $u, v \in B$,  there exists a $(u,v)$-path $(u,g,v)$, where $g \in V(G) \setminus A$. Therefore, $X$ is the mutual-visibility set of the join $G \vee H$.
\end{proof}
\begin{lemma}\label{ch2.lem5}
    If $G$ and $H$ are disjoint graphs, then $V(G)$ and $V(H)$ are mutual-visibility sets of the join $G \vee H$.
\end{lemma}
\begin{proof}
   Let $u, v \in V(G)$. If they are adjacent, they are $G$-visible. Otherwise, $d(u, v)=2$ in $G \vee H$. There exists at least one shortest $(u, v)$-path $(u,h,v)$, where $h \in V(H) $. It follows that any two vertices $u$ and $v$ of $G$ are $G$-visible. Therefore, $V(G)$ is a mutual-visibility set of the join $G \vee H$. Similar proof holds for $V(H)$ also.
\end{proof}
\begin{lemma}\label{ch2.lem7}
    Let $G$ and $H$ be two non-complete disjoint graphs with $|V(H)|\geq 2$. Let $B \subseteq V(H)$.  Then $X=V(G) \cup B$ is a mutual-visibility set of the join $G \vee H$ if and only if $B$ satisfies one of the following conditions. \\
    1. $B=\phi$.\\
2. Induced subgraph $H[B]$ is a clique.\\
        3. $B$ is a mutual-visibility set in $H$ and $diam_H(B) = 2$
\end{lemma}
\begin{proof}
Let $B\subseteq V(H) $ such that $X=V(G)\cup B$ is a mutual-visibility set of $G\vee H$. Since $V(G)$ is a mutual-visibility set of $G\vee H$ (By Lemma \ref{ch2.lem5}), $B$ may be empty. Now consider the case where $B\neq \phi$. We first claim that $diam_H(B)\leq 2$.

If $diam_H(B)>2$, then there exists $b_1, b_2 \in B$ such that $d_H(b_1, b_2) >2$. But $d_{G \vee H}(b_1, b_2) =2$, and all the shortest $(b_1, b_2)$-paths contain an element from $V(G)\subset X$. So $b_1, b_2$ are not $X$-visible. Therefore, $X$ is not a mutual-visibility set in $G \vee H$.

Now, if $diam_H(B)=0$, then $B$ must be a singleton set, which is a 1-clique in $H$. If $diam_H(B)=1$, then there are edges between every pair of vertices of $B$ and hence it is a clique. If $diam_H(B)=2$, then either $u, v \in B$ is adjacent or $d_H(u, v)=2$. In the former case, $u, v$ are $B$-visible in $H$, while in the latter case it is given that $u,v$ are $X$-visible in the join $G \vee H$. Therefore, there exists a $(u, v)$-path $P=(u, a, v)$ of length 2 in  $G \vee H$, such that $P \cap X = \{ u, v\}$. This implies $a \in V(H) \setminus B$. Since $a \notin B$,  $P \cap B = \{ u, v\}$. Therefore, $u$ and $v$ are $B$-visible in $H$. Therefore, $B$ is a mutual-visibility set in $H$.

To prove the sufficient part, assume that $B$ satisfies one of the conditions stated in the lemma and consider each of the cases separately.\\
1.  If $B=\phi$, then $X=V(G)$ and it is a mutual-visibility set of the join $G \vee H$ by Lemma \ref{ch2.lem5}.\\
2.  Assume that $H[B]$ is a clique. If $u, v \in B$ or $u \in G $ and $v \in B $ then they are adjacent in $G\vee H$. Therefore, $u, v$ are $X$-visible in $G\vee H$. If $u, v \in V(G)$ and are adjacent, they are $X$-visible. If not, then $d(u, v)=2$ in $G \vee H$. Since $H$ is not complete, $B \subsetneq V(H)$ and there exists at least one $h \in V(H)\setminus B$ and  $(u,h,v)$ is a shortest $(u, v)$-path. Therefore, $X$ is a mutual-visibility set in $G \vee H$.\\
3. Assume that $B$ is a mutual-visibility set in $H$  and $diam_H(B)=2$. Then there exists a shortest path $(b_1, h, b_2)$, where $b_1, b_2 \in B$ and $h \in V(H) \setminus B$. If $u, v \in V(G)$ are adjacent, they are $X$-visible in $G\vee H$. If not, then $d(u, v)=2$ in $G \vee H$ and the shortest $(u, v)$-path, $P=(u, h, v)$ satisfies $P \cap X=\{u, v\}$. Therefore, $u, v$ are $X$-visible.  Finally, if $u \in V(G) $ and $v \in B$ then they are adjacent in $G\vee H$ and therefore $X$-visible.  Therefore, $X$ is a mutual-visibility set in $G \vee H$.
\end{proof}
\begin{theorem}\label{ch2.th6}
     Let $G$ and $H$ be two disjoint non-complete graphs with $2 \leq m \leq n$, $m=|V(G)|$ and $n=|V(H)|$. Then, the visibility polynomial of the join of $G$ and $H$ is given by, 
     $\mathcal{V}(G \vee H)=\displaystyle \sum_{i=0}^{m+n-1}r_i x^i$
     where
     $$r_i=\left\{
  \begin{array}{ll}
  {m+n \choose i}&  \text{if }  0 \leq i \leq m \\[5pt]
 \displaystyle \sum_{k=0}^{m-1}{m \choose k}{n \choose i-k}+c_{i-m}(H) + \Theta_{i-m,2}(H)& \text{if }  m+1 \leq i \leq n \\[10pt]
 \displaystyle\sum_{k=i-n+1}^{m-1}{m \choose k}{n \choose i-k}+c_{i-m}(H) +\Theta_{i-m,2}(H)+c_{i-n}(G)+\Theta_{i-n,2}(G)  & \text{if } n+1\leq i \leq m+n-2 \\[12pt]
 c_{n-1}(H) +\Theta_{n-1,2}(H)+c_{m-1}(G)+\Theta_{m-1,2}(G)  & \text{if } i= m+n-1
  \end{array}
\right.$$
\end{theorem}
\begin{proof}
    Let $X\subseteq V(G \cup H)$ be any subset of cardinality $i$. We consider various cases based on the value of $i$, and for each case, we count the number of instances in which $X$ forms a mutual-visibility set of $G\vee H$ of cardinality $i$. Note that $V(G\cup H)$ cannot form a mutual-visibility set in $G\vee H$ by Corollary \ref{ch2.coro1}. Hence, the value of $i$ can be at most $m+n-1$ for $X$ to be a mutual-visibility set.\\  
Case 1: Let $1\leq i \leq m$. Let $X=A\cup B$, where $A \subseteq V(G), B \subseteq V(H)$ . In this case, the set $X$ can take exactly two possible forms. The first occurs when $A \subsetneq V(G)$ and $B \subsetneq V(H)$; the second when $A = V(G)$ and $B = \phi$; and the third when $A = \phi$ and $B = V(H)$. In the first case, $X$ is a mutual-visibility set by Lemma \ref{ch2.lem6}, while in the remaining two cases, $X$ is a mutual-visibility set by Lemma \ref{ch2.lem5}. Therefore, every subset of $ V(G\cup H)$ of cardinality $i$ is a mutual-visibility set, and the number of such sets is ${m+n\choose i}$.\\
Case 2:  Let $m+1\leq i \leq n$. In this case, the mutual-visibility set $X=A \cup B$ can take exactly three possible forms. The first form corresponds to the case where $A \subsetneq V(G)$; in this scenario, $ \sum_{k=0}^{m-1}{m \choose k}{n \choose i-k}$ mutual-visibility sets by Lemma \ref{ch2.lem6}. The second form occurs when $A=V(G)$ and $H[B]$ is a clique in $H$ of cardinality $i-m$, and the third form corresponds to $A=V(G)$ and $B$ is a mutual-visibility set of cardinality $i-m$ in $H$ and $diam_H(B) = 2$. Lemma \ref{ch2.lem7} guarantees that $X$ is a mutual-visibility set in both of these cases.  Therefore, total number of mutual-visibility sets of cardinality $i$ is given by $r_i=\sum_{k=0}^{m-1}{m \choose k}{n \choose i-k}+c_{i-m}(H) + \Theta_{i-m,2}(H)$ where  $c_{i-m}(H)$ denotes the number of cliques of size  $i-m$ in $H$, and $\Theta_{i-m,2}(H)$ denotes the number of mutual-visibility sets of size $i-m$ and diameter 2 in $H$. (Note that $\Theta_{1,2}(G)=0$)  \\
Case 3:  Let $n+1\leq i \leq m+n-2$. In this case, the set $X = A \cup B$ has five distinct forms. Three of these forms coincide with those described previously in case 2. The fourth form occurs when $B = V(H)$ and $G[A]$ is a clique in $G$ with cardinality $i - n$. The fifth form arises when $B = V(H)$ and $A$ is a mutual-visibility set of $G$ having cardinality $i-n$ with $diam_G(A) = 2$.
 Hence, the total number of mutual-visibility sets of cardinality $i$ is given by $r_i=\sum_{k=i-n+1}^{m-1}{m \choose k}{n \choose i-k}+c_{i-m}(H) +\Theta_{i-m,2}(H)+c_{i-n}(G)+\Theta_{i-n,2}(G)$. Note that we consider the situation $B = V(H)$ exclusively in the fourth and fifth forms; hence, in the summation term, the index $k$ begins from $2$ instead of $1$.\\
    Case 4:  Let $i=m+n-1$. The proof of this case is analogous to that of Case 3, with the key distinction that the summation term corresponds to subsets where $A\subsetneq V(G)$ is no longer applicable. 
\end{proof}
To demonstrate the theorem, consider the graphs $G$ and $H$ in Figure \ref{join}, both of which are not complete.

\begin{figure}[h]
    \centering
\begin{tikzpicture}
  \begin{scope}[xshift=3.5cm, rotate=10, scale=0.9]
    \def\n{6}
    \def\radius{1.5}
    \foreach \i in {1,...,\n} {    
      \node[circle, fill=black, inner sep=2pt] (R\i) at ({360/\n * (\i -1)}:\radius) {};
    }
    \foreach \i in {1,...,\n} {
      \pgfmathtruncatemacro\nexti{mod(\i,\n) + 1}
      \draw[line width=0.8pt] (R\i) -- (R\nexti);
    }
\node at (0,0) {$H$};
  \end{scope}
  \begin{scope}[rotate=-19]
  \def\m{4}
  \def\arcRadius{1.2}
  \def\startAngle{135}
  \def\endAngle{-135}
  \foreach \j [evaluate=\j as \angle using \startAngle + (\endAngle - \startAngle)*(\j-1)/(\m-1)] in {1,...,\m} {    
    \node[circle, fill=black, inner sep=2pt] (L\j) at (\angle:\arcRadius) {};
  }
  \foreach \j in {1,...,3} {
    \pgfmathtruncatemacro\nextj{\j + 1}
    \draw[line width=0.8pt] (L\j) -- (L\nextj);
  }
  \node at (-8mm,-2mm) {$G$};
  \draw[line width=0.8pt] (L1) -- (L3);
\end{scope}
\end{tikzpicture}
 \caption{The graphs $G$ and $H$}
    \label{join}
\end{figure}
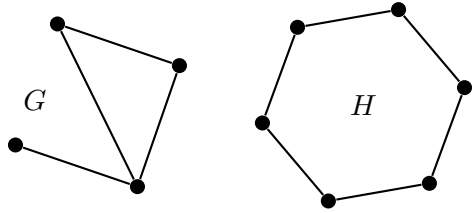
Here $m=4$ and $n=6$, and the visibility polynomial of the join of $G$ and $H$ is $\mathcal{V}(G\vee H)=\sum_{i=0}^9 r_i x^i$. For $0 \leq i \leq 4$, the number of mutual-visibility sets of size $i$ is given by $r_i={10\choose i}$. For $5\leq i \leq 9$, $r_i$ are calculated as follows: \\$r_5=\sum_{k=0}^3{4 \choose k}{6\choose 5-k}+c_1(C_6)+\Theta_{1,2}(C_6)=246+6+0=252$\\
$r_6=\sum_{k=0}^3{4 \choose k}{6\choose 6-k}+c_2(C_6)+\Theta_{2,2}(C_6)=195+6+6=207$\\
$r_7=\sum_{k=2}^3{4 \choose k}{6\choose 7-k}+c_3(C_6)+\Theta_{3,2}(C_6)+c_{1}(G)+\Theta_{1,2}(G)=96+0+2+4+0=102$\\
$r_8=\sum_{k=3}^3{4 \choose k}{6\choose 8-k}+c_4(C_6)+\Theta_{4,2}(C_6)+c_{2}(G)+\Theta_{2,2}(G)=24+0+0+4+2=30$\\
$r_9=c_5(C_6)+\Theta_{5,2}(C_6)+c_{3}(G)+\Theta_{3,2}(G)=0+0+1+1=2$\\
Some of the instances that are counted to $r_7$ and $r_9$ are depicted in Figure \ref{join_1}.

\begin{figure}[h]
  \centering
  
 \begin{subfigure}{0.45\textwidth}
    \centering
\begin{tikzpicture}

  \begin{scope}[xshift=3.5cm, rotate=10,scale=0.9]
    \def\n{6}
    \def\radius{1.5}
    \foreach \i in {1,...,\n} {
      \ifnum\i=1
        \def\mycolor{red}
      \else\ifnum\i=3
        \def\mycolor{red}
      \else\ifnum\i=5
        \def\mycolor{red}
      \else
        \def\mycolor{black}
      \fi\fi\fi
    
      \node[circle, fill=\mycolor, inner sep=2pt] (R\i) at ({360/\n * (\i -1)}:\radius) {};
    }
  
    \foreach \i in {1,...,\n} {
      \pgfmathtruncatemacro\nexti{mod(\i,\n) + 1}
      \draw[line width=0.8pt] (R\i) -- (R\nexti);
    }
  \end{scope}

  \begin{scope}[rotate=-19]
  \def\m{4}
  \def\arcRadius{1.2}
  \def\startAngle{135}
  \def\endAngle{-135}
  
  \foreach \j [evaluate=\j as \angle using \startAngle + (\endAngle - \startAngle)*(\j-1)/(\m-1)] in {1,...,\m} {
    \ifnum\j=4
      \def\mycolor{red}
    \else
      \def\mycolor{red}
    \fi

    \node[circle, fill=\mycolor, inner sep=2pt] (L\j) at (\angle:\arcRadius) {};
  }

  \foreach \j in {1,...,3} {
    \pgfmathtruncatemacro\nextj{\j + 1}
    \draw[line width=0.8pt] (L\j) -- (L\nextj);
  }
  
  \draw[line width=0.8pt] (L1) -- (L3);

  \foreach \li in {1,...,4} {
    \foreach \ri in {1,...,6} {
      \draw[dotted, line width=0.6pt] (L\li) -- (R\ri);
    }
  }
\end{scope}
\end{tikzpicture}
 \caption{$r_7$: $A=V(G)$, $B=$ MV set in $H$ with diam 2}
  \end{subfigure}
\hspace{2mm}
\begin{subfigure}{0.45\textwidth}
    \centering
\begin{tikzpicture}

  \begin{scope}[xshift=3.5cm, rotate=10, scale=0.9]
    \def\n{6}
    \def\radius{1.5}

    \foreach \i in {1,...,\n} {
      \ifnum\i=1
        \def\mycolor{red}
      \else\ifnum\i=4
        \def\mycolor{red}
      \else
        \def\mycolor{red}
      \fi\fi
    
      \node[circle, fill=\mycolor, inner sep=2pt] (R\i) at ({360/\n * (\i -1)}:\radius) {};
    }

    \foreach \i in {1,...,\n} {
      \pgfmathtruncatemacro\nexti{mod(\i,\n) + 1}
      \draw[line width=0.8pt] (R\i) -- (R\nexti);
    }
  \end{scope}

  \begin{scope}[rotate=-19]

  \def\m{4}
  \def\arcRadius{1.2}
  \def\startAngle{135}
  \def\endAngle{-135}

  \foreach \j [evaluate=\j as \angle using \startAngle + (\endAngle - \startAngle)*(\j-1)/(\m-1)] in {1,...,\m} {
    \ifnum\j=4
      \def\mycolor{black}
    \else
      \def\mycolor{red}
    \fi

    \node[circle, fill=\mycolor, inner sep=2pt] (L\j) at (\angle:\arcRadius) {};
  }
  \foreach \j in {1,...,3} {
    \pgfmathtruncatemacro\nextj{\j + 1}
    \draw[line width=0.8pt] (L\j) -- (L\nextj);
  }
  \draw[line width=0.8pt] (L1) -- (L3);
  \foreach \li in {1,...,4} {
    \foreach \ri in {1,...,6} {
      \draw[dotted, line width=0.6pt] (L\li) -- (R\ri);
    }
  }
\end{scope}
\end{tikzpicture}
 \caption{$r_9$: $G[A]=$ Clique in $G$, $B=V(H)$ }
  \end{subfigure}
    \caption{Mutual-visibility sets in $G\vee H$}
  \label{join_1}
\end{figure}
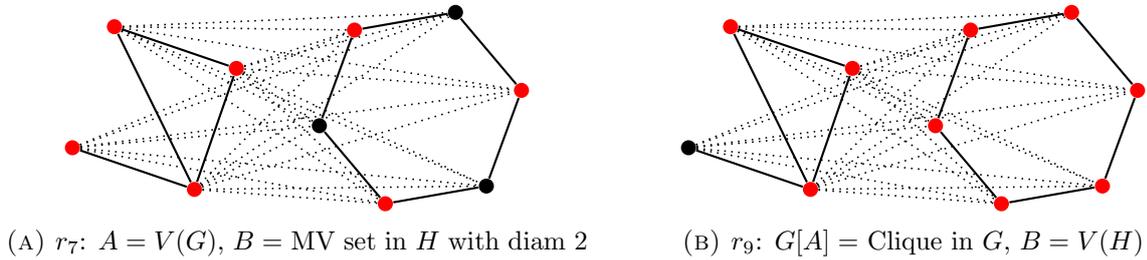
\textbf{Remark:} Since $K_{m,n}\cong \overline{K_m} \vee \overline{K_n}$, letting $G=\overline{K_m}$ and $H=\overline{K_n}$, in Theorem \ref{ch2.th6}, the result of Theorem \ref{ch2.th5} can be deduced using Vandermonde identity in combinatorics.

\section{Time Complexity of Visibility Polynomial Problem}
Di Stefano \cite{Stefano} proposed an algorithm, referred to as $MV$, which verifies in polynomial time whether a given subset $P$ of vertices in a graph $G(V, E)$ forms a mutual-visibility set.
\begin{theorem}[\cite{Stefano} Th. 3.2]\label{ch2.Th4}
    Algorithm MV solves MUTUAL-VISIBILITY TEST in $O (| P | (| V | + | E| ))$ time.
\end{theorem}
\begin{defn}\textbf{Visibility polynomial Problem:} Given a graph $G$ of order $n$, find the visibility polynomial of $G$.
\end{defn}

\begin{algorithm}[H]\LinesNumbered
\caption{V$\_$Polynomial}
\KwIn{A graph $G=(V, E)$}
\KwOut{The vector of coefficients of visibility polynomial}

$n \gets$ number of vertices of $G$\;
Initialize Counts$[0 \dots n] \gets [1,0,0,\dots,0]$\;
\For{$k \gets 1$ to $n$}{
  \For{each subset $P \subseteq V$ with $|P|=k$}
{
    \If{MV($G$, $P$) is \texttt{True}}{
      Counts[$k$] $\gets$ Counts[$k$] + 1\;
    }
  }
}
\Return Counts\;
\end{algorithm}

\begin{theorem}
    The Algorithm V$\_$Polynomial solves visibility polynomial problem in $O\bigl(|V|\,(|V|+|E|)\,2^{|V|}\bigr)$ time.
\end{theorem}
\begin{proof}
The running time of the algorithm $MV$ on the input graph $G$ and subset $P$ of $V$ is $ O\bigl(\,|P| (|V| + |E|)\bigr)$, by Theorem \ref{ch2.Th4}. Let $n = |V|$ and $m = |E|$. The algorithm $MV$ iterates over all subsets $P$ of size $k$, for each $k = 1,\ldots,n$. For each fixed $k$, there are $\binom{n}{k}$ subsets of size $k$. Thus, the total work over all subsets is:
$$
\sum_{k=1}^{n} \binom{n}{k} O\bigl(k\,(n + m)\bigr) = O(n+m) \sum_{k=1}^{n} \binom{n}{k}\,k = O(n+m)n\,2^{n-1}
$$
 Therefore, the total time complexity is 
$O\bigl(|V|(|V|+|E|)\,2^{|V|}\bigr).$
\end{proof}
\textbf{Concluding remarks:}
Since the algorithm for computing the visibility polynomial of a graph has a time complexity of $O(n^32^n)$, it becomes computationally intensive even for graphs of moderate size, underscoring the importance of theoretical investigations and structural characterizations. The distinctive properties of classes of non-isomorphic graphs sharing the same visibility polynomial can be investigated. It would be interesting to investigate the relationship between the roots of the visibility polynomial and the structural properties of the corresponding graph. \\

\textbf{Data availability:} All non-isomorphic graphs of order upto nine were generated using the nauty27r3 software package.\\

\end{document}